\documentclass[reqno,11pt]{amsart}

\usepackage{euscript}
\usepackage{enumerate}
\usepackage{amsmath}
\usepackage{mathrsfs}
\usepackage{mathtools}
\usepackage[initials]{amsrefs}
\usepackage{amssymb}
\usepackage{comment}
\usepackage{color}

\theoremstyle{plain}
\newtheorem{thm}{Theorem} [section]
\newtheorem{prop}[thm]{Proposition}
\newtheorem{cor}[thm]{Corollary}
\newtheorem{ques}[thm]{Question}
\newtheorem{lemma}[thm]{Lemma}

\newtheorem{thmsubs}{Theorem}[subsection]

\newtheorem{propsubs}[thmsubs]{Proposition}

\theoremstyle{remark}

\theoremstyle{definition}
\newtheorem{defi}[thm]{Definition}
\newtheorem{example}[thm]{Example}
\newtheorem{defisubs}[thmsubs]{Definition}

\newcommand\Afr{{\mathfrak{A}}}

\newcommand\Cpx{{\mathbb C}}

\newcommand\diag{\operatorname{diag}}
\newcommand\DT{\operatorname{DT}}
\newcommand\eps{{\epsilon}}

\newcommand\HEu{{\EuScript H}}                   

\newcommand\Mcal{{\mathcal{M}}}

\newcommand\mut{{\tilde{\mu}}}
\newcommand\Nats{{\mathbb N}}
\newcommand\oneh{{\hat 1}}

\newcommand\RealPart{{\mathrm{Re}\,}}

\newcommand\Sc{{\mathcal{S}}}

\newcommand\Tr{{\mathrm{Tr}}}

\newcommand\xh{{\hat x}}
\newcommand\norm[1]{\ensuremath{\left\vert\left\vert #1 \right\vert\right\vert}}

\begin{document}

\title[Angles and spectrality]{Angles between Haagerup--Schultz projections and spectrality of operators}

\author[Dykema]{Ken Dykema$^*$}
\address{Ken Dykema, Department of Mathematics, Texas A\&M University, College Station, TX, USA.}
\email{ken.dykema@math.tamu.edu}
\thanks{\footnotesize ${}^{*}$ Research supported in part by NSF grant DMS--1800335.}
\author[Krishnaswamy-Usha]{Amudhan Krishnaswamy-Usha$^*$ $^\dagger$}
\thanks{\footnotesize ${}^{\dagger}$ Portions of this work are included in the thesis of A. Krishnaswamy-Usha for partial fulfillment of the requirements to obtain a Ph.D. degree at Texas A\&M University}
\address{Amudhan Krishnaswamy-Usha, Department of Mathematics, Texas A\&M University, College Station, TX, USA. }
\email{amudhan@math.tamu.edu}



\begin{abstract}
We investigate angles between Haagerup--Schultz projections of operators belonging to finite von Neumann algebras,
in connection with a property analogous to Dunford's notion of spectrality of operators.
In particular, we show that an operator is similar to the sum of a normal and an s.o.t.-quasinilpotent operator that commute if and only if
the angles between its Haagerup--Schultz projections are uniformly bounded away from zero (and we call this the uniformly non-zero angles property).
Moreover, we show that spectrality is equivalent to this uniformly non-zero angles property plus decomposability.
Finally, using this characterization,
we construct an easy example of an operator which is decomposable but not spectral, and we show that Voiculescu's circular operator is not spectral
(nor is any of the circular free Poisson operators).
\end{abstract}

\date{October 24, 2020}

\maketitle

\section{Introduction}

The existence of the Jordan canonical form for an $n\times n$ complex matrix $T$ amounts to writing $T$ as a sum of
a diagonalizable operator plus a commuting nilpotent operator.
Equivalently, it implies that $T$ is similar to a normal operator plus a commuting nilpotent operator.
In 1954, Dunford~\cite{D54} introduced and studied {\em spectral} operators, which are operators $T$ on a Banach space that admit idempotent-valued spectal measures
commuting with $T$ that behave well with respect to the spectrum.
(These definitions are briefly recalled in Section~\ref{sec:spectral}, below).
With the help of a result of Wermer~\cite{W54}, he showed that on a Hilbert space, this amounts to $T$ being similar to the sum of a normal operator
and a commuting quasinilpotent operator.

Let $\Mcal$ be a von Neumann algebra equipped with a normal, faithful, tracial state $\tau$.
In this paper, we study operators $T$ belonging to $\Mcal$.
The Brown measure of $T$ is a sort of spectral distribution measure.
In~\cite{HS09}, Haagerup and Schultz proved existence of projections onto hyperinvariant subspaces of $T$ that behave well with respect to Brown measure:
for each Borel subset $B\subseteq\Cpx$, there is a {\em Haagerup--Schultz} projection $P(T,B)$.
(See Section~\ref{subsec:HS} for a brief summary of these and some other related results.)
In~\cite{DSZ15}, the Haagerup--Schultz projections were used to prove a Schur type upper triangularization result.
In particular, it was proved that every $T\in\Mcal$ is can be written as a normal operator plus an s.o.t.-quasinilpotent operator, which has as
hyperinvariant subspaces certain spectral projections of the normal operator.
An s.o.t.-quasinilpotent operator is one whose Brown measure is concentrated at $0$.

The main theme of this paper is angles between Haagerup--Schultz projections.
\begin{defi}
 Let $V$, $W$ be closed non-zero subspaces of a Hilbert space $\HEu$. Then, the angle between them is
 \[
 \alpha(V,W) := \inf \left\{ \cos^{-1}\left( \left| \left\langle v , w \right\rangle \right| \right) \mid  v \in V,\, w \in W,\,\norm{v}=\norm{w}=1\right\}     
 \]
If $p$ and $q$ are projections in $B(\HEu)$, then we let $\alpha(p,q)=\alpha(p\HEu,q\HEu)$.
\end{defi}
We say that $T\in\Mcal$ has the {\em uniformly non-zero angle property} (or UNZA property), if there is $\kappa>0$ such that for all Borel sets $B\subseteq\Cpx$,
we have $\alpha(P(T,B),P(T,B^c))\ge\kappa$.
We show that the UNZA property is analogous to spectrality for elements of finite von Neumann algebras.
In particular (Theorem~\ref{thm:unzaspec}), we show that $T\in\Mcal$ has the UNZA property if and only if it is similar in $\Mcal$ to an element of the form $N+Q$
where $N$ is normal and $Q$ is s.o.t.-quasinilpotent and commutes with $N$.
We also show (Corollay~\ref{cor:specdec}) that $T$ is spectral if and only if it is decomposable and has the UNZA property.

We should note that this connection between spectrality of operators and angles between certain of their associated subspaces is not the first.
In~\cite{D58}, Dunford provided a set of four conditions (A)-(D), which are together equivalent to spectrality.
As noted by Stampfli in~\cite{S69}, condition (B) translates to saying that the angle between local spectral subspaces is uniformly bounded away from zero.
However, although conditions (A) and (C) in Dunford's result are natural (they are now known as the single-valued extension property
and Dunford's property (C) ), condition (D) is not.
Moreover, it is not clear if properties (A), (C) and (D) together imply decomposability.

We go on to apply this characterization of spectrality for elements of finite von Neumann algebras in terms of the UNZA property to particular cases.
It is easy to construct a direct sum of matrices that is decomposable but fails the UNZA property and is, thus, not spectral.
We also show that Voiculescu's circular operator $Z$ (which was known, from~\cite{DH04}, to be decomposable) fails to have the UNZA property.
In fact, we show (Theorem~\ref{thm:circnonspec})
that for some Borel set $B$, $\alpha(P(Z,B),P(Z,B^c))=0$, and the same whenever $Z$ is a circular free Poisson operator
(a class which includes the circular operator).
We do this by explicitly constructing vectors in the Haagerup--Schultz subspaces whose angles approach zero.

Here is a brief summary of the contents of the paper:
In Section~\ref{sec:prelims}, we review some topics and earlier results that we will need.
In Section~\ref{sec:spectral}, we consider Dunford's notions of spectral and scalar type operators in the context of finite von Neumann algebras.
In Section~\ref{sec:angles}, we prove several results about the angles between Haagerup--Schultz projections, including our characterizations mentioned above.
In Section~\ref{sec:circ}, we exhibit a direct sum of matrices that fails the UNZA property and we show that Voiculescu's circular operator also fails to have the UNZA property;
thus, these operators are not spectral.

\smallskip

\noindent
{\bf Acknowledgement:}   The authors thank L\'aszl\'o Zsid\'o for inspiring conversations
and an anonymous referee for helpful suggestions.

\section{Preliminaries}
\label{sec:prelims}

\subsection{Notation}

$\Cpx$ denotes the complex plane, and $\Afr$ is its Borel $\sigma$-algebra.
Given $0\le r\le s\le\infty$, we write
\[
A_{r,s}=\{z\in\Cpx\mid r\le|z|\le s\}
\]
for the closed annulus centered at the orgin with radii $r$ and $s$.
Thus, when $r=0<s$ this is the closed ball of radius $s$, when $0<r=s<\infty$ this is the circle of radius $r$
and when $r<s=\infty$ this is the complement of the open ball of radius $r$.

Throughout, $\Mcal$ will refer to a von Neumann algebra having a normal, faithful tracial state $\tau$, and acting faithfully on a Hilbert space. 
Oftentimes, this Hilbert space will be taken to be $L^2(\Mcal,\tau)$, which is the completion of $\Mcal$ with respect to the norm $\|x\|_2:=\tau(x^*x)^{1/2}$.
We let
$\xh\in L^2(\Mcal,\tau)$ denote the element corresponding to $x\in\Mcal$.

For $T \in \Mcal$, $\sigma(T)$ will denote its spectrum, $\mu_T$ will denote its Brown measure, and for $B \in \Afr$, $P(T,B)$ will denote the corresponding Haagerup--Schultz projection.
By projection, we mean a bounded self-adjoint idempotent. 

\subsection{Brown measure and Haagerup--Schultz projections}
\label{subsec:HS}

L.Brown~\cite{B83} showed that there exists a generalization of the spectral distribution measure to non-normal operators in tracial von Neumann algebras:
\begin{thmsubs}\label{thm:brownmeas}
Let $T \in \Mcal$. Then there exists a unique probability measure $\mu_T$ such that for every $\lambda \in \Cpx$, 
\[
\int_{[0,\infty)}\log(x) d\mu_{|T-\lambda|} (x) = \int_\Cpx\log|z-\lambda| d\mu_T(z),
\]
where for a positive operator $S$, $\mu_S$ denotes the spectral distribution measure $\tau \circ E$, where $E$ is the spectral measure for $S$.
\end{thmsubs}
The measure $\mu_T$ is called the {\em Brown measure} of $T$. If $T$ is normal, $\mu_T$ equals the spectral distribution measure of $T$. 

\smallskip
Haagerup and Schultz in~\cite{HS09} constructed a set of invariant projections for $T$, which behave well with the Brown measure:
\begin{thmsubs}\label{thm:hsproj}
Let $T \in \Mcal$. For any $B \in \Afr$, there exists a unique projection $p=P(T,B) \in \Mcal$ such that 
\begin{enumerate}[(i)]
\item $Tp = pTp$ 
\item $\tau(p) = \mu_T(B)$
\item when $p\neq 0$, considering $pTp$ as an element of $p\Mcal p$, its Brown measure is concentrated in $B$
\item when $p \neq 1$, considering $(1-p)T(1-p)$ as an element of $(1-p)\Mcal (1-p)$, $\mu_{(1-p)T}$ is concentrated in $\Cpx \setminus B$. 
\item If $q \in \Mcal$ is a $T$-invariant projection such that $\mu_{Tq}$ (computed in the corner $q\Mcal q$) is concentrated in $B$, then $q \leq p$.
\end{enumerate}
Moreover, $P(T,B)$ is $T$-hyperinvariant, and if $B_1 \subset B_2$, then $P(T,B_1) \leq P(T,B_2)$.
\end{thmsubs}

In Theorem 8.1 of~\cite{HS09}, Haagerup and Schultz also show the following convergence result:

\begin{thmsubs}\label{thm:sotq}
Let $T \in \Mcal$. Then the sequence $|T^n|^{1/n}$ has a strong operator limit $A$, and for every $r\geq 0$, the spectral projection of $A$ associated with the interval
$[0,r]$ is $P(T,r\overline{\mathbb{D}})$, where $\mathbb{D}$ is the open disc of radius $1$. 
\end{thmsubs}

It follows that $\mu_T = \delta_0$ is the point mass at $0$ if and only if $|T^n|^{1/n}$ converges to $0$ in the strong operator topology.
Such operators are called {\em s.o.t.-quasinilpotent}. 

The following result is from the essential construction, found in~\cite{HS09}, which Haagerup and Schultz used to build $P(T,B)$ for general Borel sets $B$.
\begin{propsubs}\label{prop:disks}
Let $r>0$.
Suppose $\Mcal$ acts on the Hilbert space $\HEu$ and $T\in\Mcal$.
Then
\[
P(T,A_{0,r})\HEu=\{\xi\in\HEu\mid \exists \xi_n\in\HEu,\,\lim_{n\to\infty}\|\xi_n-\xi\|=0,\,\limsup_{n\to\infty}\|T^n\xi_n\|^{1/n}\le r\}
\]
and
\begin{multline*}
P(T,A_{r,\infty})\HEu= \\
\{\eta\in\HEu\mid \exists \eta_n\in\HEu,\,\lim_{n\to\infty}\|T^n\eta_n-\eta\|=0,\,\limsup_{n\to\infty}\|\eta_n\|^{1/n}\le\frac1r\}.
\end{multline*}
\end{propsubs}

The Haagerup--Schultz projections satisfy nice lattice properties, as shown in~\cite{S06}:
\begin{thmsubs}\label{thm:hsproj-lattice}
Let $B_1,B_2,... \in \Afr$.
Then 
\begin{align*}
\bigvee_{n=1}^\infty P(T,B_n)&= P \left( T, \bigcup_{n=1}^\infty B_n \right) \\
\bigwedge_{n=1}^\infty P(T,B_n)&= P \left( T, \bigcap_{n=1}^\infty B_n \right)
\end{align*}
\end{thmsubs} 

They also behave well with respect to compressions and similarities (Theorem 2.4.4, Theorem 12.3 in~\cite{CDSZ17}):
\begin{thmsubs}\label{thm:hsproj-sim}
Let $Q \in \Mcal$ be a non-zero $T$-invariant projection, and suppose $A\in \Mcal$ is invertible.  Then, for all $B \in \Afr$, we have 
\begin{enumerate}[(i)]
\item $P(T,B) \wedge Q = P^{(Q)}(TQ,B)$,
\item  $\mu_{ATA^{-1}} = \mu_T$,
\item $P(ATA^{-1},B)\HEu = AP(T,B)\HEu$,
\end{enumerate}
where $P^{(Q)}$ denotes the Haagerup--Schultz projection computed in the compression $Q\Mcal Q$. 
\end{thmsubs}

Joint Brown measures and Haagerup--Schultz projections can also be defined for commuting tuples of operators. (See~\cite{S06} and~\cite{CDSZ17}). 
\begin{thmsubs}\label{thm:brown-comm}
Let $S,T\in \Mcal$ be commuting operators. Then, there exists a unique compactly supported Borel probability measure $\mu_{S,T}$ on $\Cpx^2$ such that, for all $\lambda_i \in \Cpx$, 
\[ 
\tau(\log| \lambda_1 S + \lambda_2 T - 1|) = \int_{\Cpx^2} \log|\lambda_1 z + \lambda_2 w -1| d\mu_{S,T}(z,w).
\]
\end{thmsubs}

\begin{thmsubs}\label{thm:hsproj-comm}
For commuting operators $S,T \in \Mcal$, and a Borel set $B \subset \Cpx^2$, there is a projection 
$P((S,T):B) \in \Mcal$ which is $(S,T)$-hyperinvariant, and which satisfies the following:
\begin{enumerate}[(i)]
\item For $B_1,B_2 \subset \Cpx$, $P((S,T):B_1 \times B_2) = P(S,B_1) \wedge P(T,B_2) $
\item $P((S,T):\cdot)$ satisfies lattice properties analogous to Theorem \ref{thm:hsproj-lattice}. 
\item For a Borel set $B \subset \Cpx^2$, with $p=P((S,T):B)$, if $0<p<1$, the Brown measure  of $(Sp,Tp)$ and $((1-p)S,(1-p)T)$, computed in the compressions $p\Mcal p$ and $(1-p)\Mcal (1-p)$ respectively, are concentrated in $B$, and $B^c$. 
\item $\mu_{(S,T)}(B) = \tau(P((S,T):B))$.
\end{enumerate} 
\end{thmsubs}

The joint Brown measures and Haagerup--Schultz projections behave well under pushforwards. In particular, (Remark 6.5 in~\cite{S06}):

\begin{propsubs}\label{prop:hscomm}
Let $S,T \in \Mcal$ be commuting operators. Let $a:\Cpx^2 \to \Cpx$ denote the addition map. Then, for any $B \in \Afr$, we have 
\[ 
P(S+T,B) = P( (S,T) : a^{-1}(B) ).
\]
Hence, if $T$ is s.o.t.-quasinilpotent, then $P(S+T,B)=P(S,B)$.
\end{propsubs}

\subsection{Decomposability of operators}

Decomposability of operators was introduced by Foia\c{s}~\cite{F63} and studied by many authors, including Apostol~\cite{A68}.
See the book~\cite{LN00} of Laursen and Neumann for more.

\begin{defisubs}
An operator is said to be {\em decomposable} if, for every pair $(U,V)$ of open sets in the plane whose union is $\Cpx$, there are closed,
$T$-invariant subspaces $\HEu'$ and $\HEu''$ such that $\HEu =\HEu' + \HEu''$
and such that the restriction of $T$ to those have spectra contained in $U$ and $V$ respectively.
\end{defisubs}

In a tracial von Neumann algebra, we have the following equivalent formulation (Proposition 3.1 in~\cite{DNZ18}):

\begin{propsubs}\label{prop:decom}
Let $T \in \Mcal$. Then the following are equivalent:
\begin{enumerate}[(i)]
\item $T$ is decomposable.
\item For all $B \in \Afr$,
\begin{align*}
\sigma(TP(T,B))&\subset \overline{B},  \textrm{ if } P(T,B)\neq 0\\
\sigma\left((1-P(T,B^c))T\right)&\subset \overline{B}, \textrm{ if } P(T,B)\neq 1
\end{align*}
where the spectra are computed in the compressions of $\Mcal$ by $P(T,B)$ and $1-P(T,B^c)$ respectively.  
\end{enumerate}
As a corollary, the support of the Brown measure of a decomposable operator is equal to its spectrum. 
\end{propsubs}

The local spectral subspaces of an operator play an important role in decomposability.
We will not go into details here, (see~\cite{LN00} for more information), but we note the following result of Haagerup and Schultz (Proposition~9.2 of~\cite{HS09}),
which we will use.
\begin{propsubs}\label{prop:HSdecomp}
Suppose $T\in\Mcal$ is decomposable.
Then for every $B \in \Afr$, the range of $P(T,B)$ is the closure of the local spectral subspace $\HEu_T(B)$.
\end{propsubs}

An operator is {\em strongly decomposable} if its restriction to every local spectral subspace is decomposable. For $T \in \Mcal$, this is equivalent to $TP(T,B)$ being decomposable, for every $B \in \Afr$.

\subsection{R-diagonal operators}
The R-diagonal operators were first introduced and studied by Nica and Speicher~\cite{NS97} and are natural objects in free probability theory.
In a finite von Neumann algebra, 
an R-diagonal operator $x$ is one that has the same $*$-distribution as $uh$, where $u$ is a Haar unitary, $h=|x|$ is positive, and the pair $(u,\,h)$ is $*$-free.

Here we collect some results and observations of Haagerup and Larsen~\cite{HL01}:
\begin{propsubs}\label{prop:Rdiag}
Suppose $x\in\Mcal$ is R-diagonal.
\begin{enumerate}[(i)]
\item\label{it:RdInv} if $x$ is invertible, then also $x^{-1}$ is R-diagonal,
\item\label{it:Rd2nm} for every $k\in\Nats$,
$\|x^k\|_2=\|x\|_2^{k}$,
\item\label{it:RdSpecRad} the spectral radius of $x$ equals $\|x\|_2$.
\end{enumerate}
\end{propsubs}
\begin{proof}
Assertions~\eqref{it:RdInv} and~\eqref{it:Rd2nm} are from Proposition 3.10 of~\cite{HL01}, while the assertion~\eqref{it:RdSpecRad}
follows from Proposition 4.1 of~\cite{HL01} and the fact that $x$ R-diagonal implies that $x$ has the same $*$-distribution as $vx$ when $v$ is a Haar unitary
that is $*$-free from $x$.
\end{proof}

\subsection{DT-operators}
In~\cite{DH04}, the first author and Uffe Haagerup introduced the class of DT-operators and proved that they are all strongly decomposable.
For each compactly supported Borel probability measure $\mu$ on $\Cpx$ and each $c>0$, there is a $\DT(\mu,c)$ operator $Z$,
(or, more correctly, there is a $\DT(\mu,c)$ $*$-distribution, and every element of a $W^*$-noncommutative probability space having this $*$-distribution is called
a $\DT(\mu,c)$ operator).
This operator can be realized as $Z=D+cT$, where $D$ is a normal operator and $T$ is the ``upper triangular half'' of a semicircular operator that is free from an abelian
algebra containing $D$.
See~\cite{DH04} for details.

For convenience, we collect some results (or easy observations) from~\cite{DH04}:
\begin{propsubs}\label{prop:DT}
Suppose $Z$ is a $DT(\mu,c)$ operator.
\begin{enumerate}[(i)]
\item\label{it:DTmult} if $w\in\Cpx\setminus\{0\}$, then $wZ$ is a $\DT(\mu\circ m_{w^{-1}},|w|c)$ operator, where $m_{w^{-1}}$ is the set map
of multiplication by $w^{-1}$,
\item\label{it:DTspec} the spectrum of $Z$ equals the support of $\mu$,
\item\label{it:DTBrown} the Brown measure of $Z$ is $\mu$.
\end{enumerate}
\end{propsubs}

\subsection{Circular free Poisson operators}\label{subsec:cfP}
In Definition~1.1 of~\cite{DH01}, a circular free Poisson operator of parameter $c\ge1$ is defined to be
an R-diagonal operator $x$ as above such that $|x|^2$ has moments equal to those of a free Poisson distribution $\nu_c$ with paramenter $c$.
Namely, this distribution is absolutely continuous with respect to Lebesgue measure and has density 
\[
\frac{d\nu_c}{d\lambda}(t)=\frac{\sqrt{(b-t)(t-a)}}{2\pi t}1_{[a,b]}(t),
\]
where $a=(\sqrt{c}-1)^2$ and $b=(\sqrt{c}+1)^2$.

Theorem~7.3 of~\cite{DH04} shows that the DT-operators that are also R-diagonal are precisely scalar multiples of the circular free Poisson operators, 
and that a circular free Poisson operator of parameter $c$ is
a $\DT(\mu,1)$ operator,
where $\mu$ is the uniform probability measure on the annulus $A_{\sqrt{c-1},\sqrt{c}}$ centered at the origin and with radii $\sqrt{c-1}$ and $\sqrt{c}$.

\begin{propsubs}\label{prop:cfPnorms}
Let $Z$ be a circular free Poisson operator of parameter $c$.
Then
\[
\|Z\|_2=\sqrt{c}.
\]
If $c>1$, then 
\[
\|Z^{-1}\|_2=\frac1{\sqrt{c-1}}.
\]
\end{propsubs}
\begin{proof}
By Proposition~\ref{prop:Rdiag}\eqref{it:RdSpecRad}, $\|Z\|_2$ equals the spectral radius of $Z$.
By Proposition~\ref{prop:DT}, $Z$ has spectrum equal to the annulus $A_{\sqrt{c-1},\sqrt{c}}$.
Similarly, if $c>1$, then $Z$ is invertible and using Proposition~\ref{prop:Rdiag}\eqref{it:RdInv},
$\|Z^{-1}\|_2$ is the spectral radius of $Z^{-1}$.
But $Z^{-1}$ has spectrum equal to the annulus $A_{c^{-1/2},(c-1)^{-1/2}}$.
\end{proof}

\section{Spectral operators in finite von Neumann algebras}
\label{sec:spectral}

\begin{defi}\label{def:idemspecmeas}
A {\em bounded idempotent-valued spectral measure} in $\Mcal$ is a mapping $\sigma\mapsto E(\sigma)$ that assigns to every $\sigma\in\Afr$ 
an idempotent $E(\sigma)\in\Mcal$ so that
\begin{enumerate}[(i)]
\item\label{it:ivsmi} $E(\Cpx)=1$,
\item\label{it:ivsmii} for all $\sigma_1,\sigma_2\in\Afr$, $E(\sigma_1\cap\sigma_2)=E(\sigma_1)E(\sigma_2)$,
\item\label{it:ivsmiii} for all $\sigma_1,\sigma_2,\ldots\in\Afr$ such that $\sigma_i\cap\sigma_j=\emptyset$ whenever $i\ne j$,
\[
E(\bigcup_{i=1}^\infty\sigma_i)=\sum_{i=1}^\infty E(\sigma_i),
\]
where the sum converges with respect to $\|\cdot\|_2$,
\item\label{it:ivsmiv} $\sup_{\sigma\in\Afr}\|E(\sigma)\|<\infty$.
\end{enumerate}

Unless otherwise mentioned, in the rest of this article, an {\em idempotent-valued spectral measure} will refer to a {\em bounded idempotent-valued spectral measure}
\end{defi}

Of course, a bounded idempotent-valued spectral measure $E$ where each $E(\sigma)$ is self-adjoint is just called a {\em spectral measure}.

It is known that a bounded idempotent-valued spectral measure in $B(\HEu)$, is similar to a spectral measure in $B(\HEu)$.
This may be found in~\cite{M59} ({\em cf} \cite{W54}), but we have not been able to obtain a copy of~\cite{M59}.
For completeness, we provide a proof of this result, when $B(\HEu)$ is replaced with $\Mcal$.
\begin{prop}\label{prop:smsim}
Suppose $E$ is a bounded idempotent-valued spectral measure in $\Mcal$.
Then there is an invertible $A\in\Mcal$ so that for every $\sigma\in\Afr$, the idempotent $A^{-1}E(\sigma)A$ is self-adjoint.
\end{prop}
\begin{proof}
Fix a normal faithful representation $\Mcal\hookrightarrow B(\HEu)$.
Given a finite Borel partition $\pi=\{\sigma_1,\ldots,\sigma_n\}$ of $\Cpx$, we consider the sesquilinear form on $\HEu$ given by
\[
\langle x,y\rangle_\pi=\sum_{i=1}^n\langle E(\sigma_i)x,E(\sigma_i)y\rangle
\]
and denote the corresponding norm by 
\[
\|x\|_\pi=\left(\sum_{i=1}^n\|E(\sigma_i)x\|^2\right)^{1/2}.
\]
From Lemma~1 of~\cite{W54}, we have
\begin{equation}\label{eq:nmpi}
\frac1{2M}\|x\|\le\|x\|_\pi\le2M\|x\|,
\end{equation}
for every $x\in\HEu$, where $M=\sup_{\sigma\in\Afr}\|E(\sigma)\|$.

Let $\Omega$ be the directed set of all finite Borel partitions of $\Cpx$, partially ordered by refinement.
Consider the net
\begin{equation}\label{eq:sesqnet}
\Omega\ni\pi\mapsto\langle\cdot,\cdot\rangle_\pi.
\end{equation}
We identify each sesquilinear form $\langle\cdot,\cdot\rangle_\pi$ with its restriction to the Cartesian product
$\Sc_1\times\Sc_1$
of the unit sphere of $\HEu$ with itself.
Using the upper bound from~\eqref{eq:nmpi}, we have $|\langle x,y\rangle_\pi|\le 2M$ for every $(x,y)\in\Sc_1\times\Sc_1$.
Thus, each sesquilinear form $\langle\cdot,\cdot\rangle_\pi$ is identified with an element of the product space $X=\prod_{\Sc_1\times\Sc_1}2M\mathbb{D}$ of
copies of the
closed disk of radius $2M$, which is compact, by Tychonoff's Theorem.
Thus, the net~\eqref{eq:sesqnet} has an accumulation point in $X$, and this extends to a bounded sesquilinear form $\langle\cdot,\cdot\rangle_\alpha$ on $\HEu$.

Writing $\|x\|_\alpha=\langle x,x,\rangle_\alpha^{1/2}$, from~\eqref{eq:nmpi} we have
\begin{equation}\label{eq:nmalpha}
\frac1{2M}\|x\|\le\|x\|_\alpha\le2M\|x\|.
\end{equation}
Let $x,y\in\HEu$.
If $\pi=\{\sigma_1,\ldots,\sigma_n\}\in\Omega$ and $\sigma=\bigcup_{i\in I}\sigma_i$ for some $I\subseteq\{1,\ldots,n\}$, then
$\langle E(\sigma)x,y\rangle_\pi=\langle x,E(\sigma)y\rangle_\pi$.
This implies that, for every $\sigma\in\Afr$, 
\begin{equation}\label{eq:Ealpha}
\langle E(\sigma)x,y\rangle_\alpha=\langle x,E(\sigma)y\rangle_\alpha.
\end{equation}
Since $\langle\cdot,\cdot\rangle_\alpha$ is a bounded sesquilinear form, there is $A\in B(\HEu)$, $A\ge0$, so that for all $x,y\in\HEu$, we have
\[
\langle x,y\rangle_\alpha=\langle Ax, Ay\rangle.
\]
Using~\eqref{eq:nmalpha}, we see that $A$ is invertible.
From~\eqref{eq:Ealpha}, we get
$A^2E(\sigma)=E(\sigma)^*A^2$, from which we get
\[
(AE(\sigma)A^{-1})^*=AE(\sigma)A^{-1}.
\]

It remains to show $A\in\Mcal$.
Suppose $U$ is a unitary in the commutant of $\Mcal$.
We see immediately that for every $\pi\in\Omega$ and for all $x,y\in\HEu$ we have
$\langle Ux,Uy\rangle_\pi=\langle x,y\rangle_\pi$, so we must have
\[
\langle AUx,AUy\rangle=\langle Ux,Uy\rangle_\alpha=\langle x,y\rangle_\alpha=\langle Ax,Ay\rangle.
\]
Thus, $U$ commutes with $A^2$, so $A^2\in\Mcal$ and $A\in\Mcal$.
\end{proof}

\begin{defi}

Following Dunford~\cite{D54}, an operator $T\in B(\HEu)$ is called a {\em spectral operator} if there exists an idempotent-valued spectral measure $E$ such that 
\begin{enumerate}[(i)]
\setcounter{enumi}{4}
\item $E(B)T=TE(B)$, for every $B \in \Afr$. (in particular, $E(B)\HEu$ is an invariant subspace for $T$)
\item The spectrum of $T$ restricted to the range of $E(B)$ is contained in the closure of $B$:
\[
\sigma(T,E(B)\HEu) \subseteq \overline{B}.
\]
\end{enumerate}
\end{defi}

From Theorems 5 and 6 in~\cite{D54}, if $T$ is a spectral operator,
its idempotent-valued spectral measure is uniquely defined, and $E(B)$ belongs to the bicommutant of $\{T\}$, for every $B \in \Afr$.
 
\begin{defi}
An operator $S\in B(\HEu)$ is said to be of {\em scalar type} if $S$ is spectral and also satisfies the equation
\begin{equation}\label{eq:Sint}
S = \int_{\sigma(S)} \lambda E(d\lambda),
\end{equation}
where $E$ is its associated spectral measure, and the integral is in the uniform operator norm topology. 
\end{defi}

Scalar type operators can be characterised precisely as those operators which are similar to normal operators. 
\begin{thm}\label{thm:norm-scl}
Let $\Mcal$ be a von Neumann algebra.
Then $S\in \Mcal$ is a scalar type operator if and only if there exists an invertible element $A$ in $\Mcal$, such that $A^{-1}SA$ is a normal operator.
\end{thm}
\begin{proof}
Note that if $S \in \Mcal$ is of scalar type, then its idempotent-valued spectral measure actually lies in $\Mcal$.
Using the invertible element $A$ constructed in Proposition \ref{prop:smsim}, since $ A^{-1}E(\cdot)A$ defines a spectral measure, the integral
\begin{equation}\label{eq:normscl}
N = \int_{\sigma(S)} \lambda A^{-1} E(d\lambda) A = A^{-1} S A  
\end{equation}
defines a normal operator.

Conversely, if $A^{-1}SA=N$ is normal, then the map 
\[ 
E: B \mapsto AP(N,B)A^{-1}  
\]
defines an idempotent-valued spectral measure.
Clearly, $E$ behaves well with respect to the spectrum for $S$, so $S$ is a spectral operator.
Moreover, equation~\eqref{eq:normscl} still holds, so~\eqref{eq:Sint} holds and $S$ is of scalar type. 
\end{proof}

Spectral operators can be characterised by the following decomposition property (see~\cite{D54}).
\begin{prop}\label{prop:spcdec}
If $S \in \Mcal$ is a scalar type operator and $Q\in\Mcal$ is quasinilpotent with $SQ=QS$, then $T=S+Q$ is a spectral operator. 

Conversely, if $T \in \Mcal$ is a spectral operator, then $T$ can be written as $T=S+Q$, where $S,Q\in \Mcal$, $Q$ is quasinilpotent, $S$ is scalar type and $SQ=QS$.
Morever, we have
\[ 
S = \int_{\sigma(T)} \lambda E(d\lambda),
\] 
where $E$ is the idempotent-valued spectral measure associated to $T$.
\end{prop}

The Haagerup--Schultz projections of spectral operators are determined by their idempotent-valued spectral measures:
\begin{prop}\label{prop:spcproj}
Let $T \in \Mcal$ be a spectral operator with idempotent-valued spectral measure $E$. Then, for every $B \in \Afr$,
\begin{equation}\label{eq:PEH} 
P(T,B)\HEu = E(B)\HEu.
\end{equation}
\end{prop}
\begin{proof}

It is known (see Corollary 1.2.25 in ~\cite{LN00}) that for a spectral operator $T$, and a closed set $K$, the range of the the spectral measure of $K$, $E(K)\HEu$,
is equal to the local spectral subspace $\HEu_T(K)$. 

Then, since $T$ is decomposable, by (Haagerup and Schultz's result)
Proposition~\ref{prop:HSdecomp} and the fact that for decomposable operators, the local spectral subspaces
for closed sets are closed, we have
$P(T,K)\HEu=\HEu_T(K)$.
Thus, the desired equality~\eqref{eq:PEH} holds for closed sets $B$.

Now, given an arbitrary $B \in \Afr$, by inner regularity of $\mu_T$, there is an increasing family $K_1\subseteq K_2\subseteq\cdots$ of closed subsets of $B$ such that
$\mu_T(B\setminus\bigcup_{k=1}^\infty K_n)=0$.
Together with the lattice property Theorem~\ref{thm:hsproj-lattice}, this implies
\[
P(T,B)=P(T,\bigcup_{n=1}^\infty K_n)=\bigvee_{n=1}^\infty P(T,K_n).
\]
Thus, we have
\begin{equation}\label{eq:P<E}
P(T,B)\HEu=\overline{\bigcup_{n=1}^\infty P(T,K_n)\HEu}=\overline{\bigcup_{n=1}^\infty E(K_n)\HEu}\subseteq E(B)\HEu.
\end{equation}

Let $p$ and, respectively, $p'$ be the orthogonal projection from $\HEu$ onto $E(B)\HEu$ and, respectively, $E(B^c)\HEu$.
Since $E(B)E(B^c)=0$, we have $p\wedge p'=0$.
However, from~\eqref{eq:P<E} we have $P(T,B)\le p$ and, likewise, $P(T,B^c)\le p'$.
We also have
\[
\tau(P(T,B))+\tau(P(T,B^c))=\mu_T(B)+\mu_T(B^c)=1.
\]
Since $\tau(p\wedge p')\ge\tau(p)+\tau(p')-1$, we cannot have $\tau(p)+\tau(p')>1$.
Thus, we must have $\tau(p)=\tau(P(T,B))$ and $\tau(p')=\tau(P(T,B^c))$.
This implies $p=P(T,B)$, namely, that~\eqref{eq:PEH} holds.
\end{proof}

\section{Angles between Haagerup--Schultz projections}
\label{sec:angles}

The following is well-known, but we provide a proof for completeness:
\begin{lemma}\label{lem:nzasum}
Let $V,W$ be closed subspaces of $\HEu$ with $V \cap W = \{0\}$ and $\overline{V+W}=\HEu$. Then the following  are equivalent:
\begin{enumerate}[(i)]
\item\label{it:aVW} $\alpha(V,W)>0$. 
\item\label{it:V+W} $V + W$ is closed.
\item\label{it:bdidem} There exists a bounded idempotent $e \in B(\HEu)$ such that
\[
e\HEu=V\quad\text{and}\quad(1-e)\HEu=W.
\]
\end{enumerate}
Moreover, to refine the implication \eqref{it:aVW}$\implies$\eqref{it:bdidem},
there is a continuous, strictly decreasing function $f:(0,1]\to [1,\infty)$ such that
\[
\|e\|\le f\big(1-\cos(\alpha(V,W))\big).
\]
\end{lemma}
\begin{proof}
\eqref{it:aVW} implies \eqref{it:V+W}: 
Let $\eps=1-\cos(\alpha(V,W))$.
Then $\eps>0$.
For $v \in V, w \in W$, we have $|\langle v,w\rangle|\le(1-\epsilon) \|v\|\|w\|$ and, thus,
\begin{multline}\label{eq:v+w}
\| v + w \|^2 = \|v\|^2 + \|w\|^2 + 2\RealPart\langle v,w\rangle \\
\ge\|v\|^2+\|w\|^2-2(1-\eps)\|v\|\,\|w\|\ge2\eps\|v\|\,\|w\|.
\end{multline}
So either $\|v\|$ or $\|w\|$ is $\le\|v+w\|/\sqrt{2\eps}$.
If $\|v\|$ is so bounded, then
\[
\|w\|\le\|v+w\|+\|v\|\le\left(1+\frac1{\sqrt{2\eps}}\right)\|v+w\|.
\]
By symmetry, we always have
\begin{equation}\label{eq:vwbound}
\|v\|,\|w\|\le\left(1+\frac1{\sqrt{2\eps}}\right)\|v+w\|.
\end{equation}
Consider a sequence $(v_n + w_n)_{n=1}^\infty$ with $v_n \in V$ and $w_n \in W$ that converges in $\HEu$ to a vector $z$.
We will show $z\in V+W$.
Using~\eqref{eq:vwbound}, we have that the sequences $(v_n)_{n=1}^\infty$ and $(w_n)_{n=1}^\infty$ are Cauchy, hence, converge to some elements $v\in V$ and $w\in W$,
respectively.
Hence, $z=v+w\in V+W$.

\eqref{it:V+W} implies \eqref{it:bdidem}:
The map $e:V+W\to V$ which is the identity on $V$ and has kernel equal to $W$ is well defined.
By the closed graph theorem, it is bounded.

\eqref{it:bdidem} implies \eqref{it:aVW}: 
If the angle were zero, we would have unit vectors $v_n \in V$ and $w_n \in W$
such that $\langle v_n, w_n \rangle \to 1$.
Then $\|v_n-w_n\|\to0$, but
$e(v_n-w_n)=v_n$.
This contradicts that $e$ is bounded.
  
In order to bound the norm of $e$, let $\eps=1-\cos(\alpha(V,W))$.
Let $w\in W$ and $v\in V$ with $\|w\|=1$ and $\|v\|=a$.
Then, using the first inequality in~\eqref{eq:v+w}, we get
\[
\|v+w\|^2\ge(1-a)^2+2a\eps,
\]
which yields
\[
\frac{\|e(v+w)\|}{\|v+w\|}\le\sqrt{\frac{a^2}{(1-a)^2+2\eps a}}.
\]
When $0<\epsilon<1$, the right hand side attains its maximum value of 
\[
f(\eps):=\frac{1}{\sqrt{\epsilon(2-\epsilon)}}
\]
when $a=1/(1-\eps)$.
\end{proof}

We now examine angles between Haagerup--Schultz projections of disjoint closed sets.

\begin{thm}\label{thm:disjointclosed}
If $T\in\Mcal$ is decomposable and if $F_1$ and $F_2$ are closed subsets of $\Cpx$ with $F_1\cap F_2=\emptyset$, then
\[
\alpha(P(T,F_1),P(T,F_2))>0.
\]
\end{thm}
\begin{proof}
Let $G= F_1 \cup F_2$, $p=P(T,G)$, and consider the operator $Tp$.
Since $T$ is decomposable, its spectrum $\sigma_{p\Mcal p}(Tp)$ (in the compression $p\Mcal p$) is a subset of $G$.
From Theorems~\ref{thm:hsproj-sim} and~\ref{thm:hsproj-lattice}, 
\[
P^{(p)}(Tp,F_i) = P(T,F_i)
\]
Since $\sigma(Tp) \subset G$, we can apply the holomorphic functional calculus for the function $1_{F_1}$ to $Tp$
and the resulting operator 
$e=1_{F_1}(Tp)$ is a bounded idempotent.
Since
$Tp$ restricted to the range of $e$ has spectrum contained in $F_1$,
we have $e\le P(Tp,F_1)$ and we get
\begin{equation}\label{eq:epH}
e(p\HEu) \subseteq P(Tp,F_1)(p\HEu) =P(T,F_1)(p\HEu) = P(T,F_1)\HEu.
\end{equation} 
Similarly, we have
\begin{equation}\label{eq:1-epH}
(p-e)(p\HEu)=1_{F_2}(Tp)(p(\HEu))\subseteq P(T,F_2)\HEu.
\end{equation}
Since 
\[
e(p\HEu)+(p-e)(p\HEu) = p\HEu=\overline{P(T,F_1)\HEu+P(T,F_2)\HEu},
\]
we must have equality for both inclusions in~\eqref{eq:epH} and~\eqref{eq:1-epH}
and that the sum of subspaces
$P(T,F_1)\HEu+P(T,F_2)\HEu$ is closed.
By Lemma~\ref{lem:nzasum}, this implies that the two projections have non-zero angle.
\end{proof}

The next example shows that the non-zero angle conclusion of the previous theorem may fail if $T$ is not decomposable.
\begin{example}\label{ex:disjointclosed}
Consider 
\[
T=\oplus_{k=1}^\infty T_k\in\bigoplus_{k=1}^\infty M_k(\Cpx),
\]
where the algebra
on the right is the $\ell^\infty$ sum
embedded into a finite von Neumann algebra $\Mcal$
and where $T_k$ is the $k\times k$ matrix
\[
T_k=\left(
\begin{smallmatrix}
-1&1 \\
&-1&1 \\
&&\ddots&\ddots \\
&&&-1&1 \\
&&&&0
\end{smallmatrix}
\right)
\]
consisting of $(-1,\ldots,-1,0)$ on the main diagonal, all entries on the diagonal above it being $1$, and all other entries of the matrix being $0$.
Note that the Brown measure of $T$ is supported on $\{-1,0\}$, but it is easy to see that the spectrum of $T$ is the closed disk of radius $1$
centered at $-1$.
Indeed, if $d=-1-z$ for $z\in\Cpx\setminus\{-1,0\}$, then considering the $k\times k$ diagonal matrix $D_k=\diag(d,d,\ldots,d,-z)$ and the $k\times k$ Jordan
block matrix $J_k$, so that $T_k-zI_k=D_k+J_k$, we have
\[
(T_k-zI_k)^{-1}=\left(\sum_{j=0}^{k-1}(-D_k^{-1}J_k)^j\right)D_k^{-1}=\left(\sum_{j=0}^{k-1}(-d^{-1}J_k)^j\right)D_k^{-1}.
\]
Note that this stays uniformly bounded in operator norm as $k\to\infty$ if and only if $|1+z|>1$.
From this, the assertion about the spectrum of $T$ follows.
In particular, $T$ is not decomposable, since the support of its Brown measure is much smaller than its spectrum. (Proposition ~\ref{prop:decom})

The vector $v_k=(1,1,\ldots,1)^t$ lies in the kernel of $T_k$ while the vector $w_k=(1,1,\ldots,1,0)^t$ lies $\ker(T_k+I_k)^k)$
and the angle between $v_k$ and $w_k$ is $\arccos(\sqrt{1-\frac1k})$.
This implies that angle between $P(T,\{-1\})$ and $P(T,\{0\})$ is zero.
\end{example}

\begin{defi}
Let $T\in\Mcal$.
Let $P(T,B)$ denote the Haagerup--Schultz projection of $T$ corresponding to $B \in \Afr$.
We say $T$ has the {\em non-zero angle property} (or NZA property) if for every $B\in \Afr$ satisfying $P(T,B)\neq 0$ and $P(T,B^c)\neq 0$, we have
\[
\alpha(P(T,B),P(T,B^c))> 0.
\]
We say $T$ has the {\em  uniformly non-zero angle property} (or UNZA property) if there exists $\kappa>0$ such that
for every $B$ satisfying $P(T,B)\neq 0$ and $P(T,B^c)\neq 0$, we have
\[
\alpha(P(T,B),P(T,B^c))\ge\kappa.
\]
\end{defi}

\begin{ques}
If $T$ satisfies the NZA property, must it also satisfy the UNZA property?
We don't know the answer, even assuming, for example, that $T$ has countable spectrum.
\end{ques}

Now, using Haagerup--Schultz projections, we construct idempotent-valued spectral measures assuming we have the UNZA property. 

\begin{lemma}\label{lem:unza-spc}
Assume $T\in\Mcal$ has the uniformly non-zero angle property.
Then there exists an idempotent-valued spectral measure $E$ with the following properties,
where $B \in \Afr$.
\begin{enumerate}[(a)]
\item\label{it:unzaidem.a} $E(B)\HEu=P(T,B)\HEu$ and $\ker E(B)=P(T,B^c)\HEu$,
\item\label{it:TEB} $TE(B)=E(B)T$,
\item\label{it:EBBrown} The Brown measure of the restriction of  $T$ to the range of $E(B)$ is concentrated in $B$.
\end{enumerate}
\end{lemma}
\begin{proof}
By Lemma \ref{lem:nzasum}, for any $B \in \Afr$ there is a bounded idempotent $E(B)$ satisfying condition~\eqref{it:unzaidem.a}.
We verify that $E$ is indeed an idempotent-valued spectral measure by checking the conditions of Definition~\ref{def:idemspecmeas}.

Clearly, $E(\Cpx)=1$, so \ref{def:idemspecmeas}\eqref{it:ivsmi} holds.
By Lemma~\ref{lem:nzasum} and the UNZA hypothesis, we have uniform boundedness of the $E(B)$, so \ref{def:idemspecmeas}\eqref{it:ivsmiv} holds.

Note that if $B_1$ and $B_2$ are disjoint Borel subsets of $\Cpx$, then it follows from the UNZA property that $\alpha(P(T,B_1),P(T,B_2))>0$ and,
thus, by Theorem~\ref{thm:hsproj-lattice} and Lemma~\ref{lem:nzasum}, that
\begin{multline*}
P(T,B_1\cup B_2)\HEu=(P(T,B_1)\vee P(T,B_2))\HEu \\
=\overline{P(T,B_1)\HEu+P(T,B_2)\HEu}=P(T,B_1)\HEu+P(T,B_2)\HEu.
\end{multline*}
Iterating this, we see that if $B_1,\ldots,B_n$ are pairwise disjoint and $B=\bigcup_{j=1}^nB_j$, then
\begin{equation}\label{eq:PTBsum}
P(T,B)\HEu=P(T,B_1)\HEu+\cdots+P(T,B_n)\HEu.
\end{equation}

We now show that property \ref{def:idemspecmeas}\eqref{it:ivsmii} holds.
Let $B_1,B_2 \in \Afr$.
If $\xi\in\HEu$ then, by the above, we may write
$\xi=\xi_{00}+\xi_{01}+\xi_{10}+\xi_{11}$,
where
\begin{alignat*}{2}
\xi_{00}&\in P(T,B_1^c\cap B_2^c)\quad&\xi_{01}&\in P(T,B_1^c\cap B_2) \\
\xi_{10}&\in P(T,B_1\cap B_2^c)\quad&\xi_{11}&\in P(T,B_1\cap B_2).
\end{alignat*}
We have
\[
E(B_1\cap B_2)\xi_{ij}=\begin{cases}\xi_{11}&\text{if }i=j=1,\\0&\text{otherwise.}\end{cases}
\]
We also have
\begin{align*}
E(B_1)E(B_2)\xi_{00}&=E(B_1)0=0 \\
E(B_1)E(B_2)\xi_{01}&=E(B_1)\xi_{01}=0 \\
E(B_1)E(B_2)\xi_{10}&=E(B_1)0=0 \\
E(B_1)E(B_2)\xi_{11}&=E(B_1)\xi_{11}=\xi_{11}.
\end{align*}
Thus, we have
\[
E(B_1)E(B_2)\xi=\xi_{11}=E(B_1\cap B_2)\xi.
\]

We now show that property \ref{def:idemspecmeas}\eqref{it:ivsmiii} holds.
Let $B_1,B_2,\ldots \in \Afr$ be pairwise disjoint. 
Let 
$A_n=\bigcup_{i=1}^nB_i$ and $A=\bigcup_{i=1}^\infty B_i$.
Given $\xi\in\HEu$ and $n\in\Nats$, using the property proved at~\eqref{eq:PTBsum}, we may write
$\xi=\eta+\xi_1+\cdots+\xi_n$, where $\eta\in P(T,A_n^c)\HEu$ and $\xi_j\in P(T,B_j)\HEu$.
For each $j$, we have $E(B_j)\xi=\xi_j$ and
\[
E(A_n)\xi=\xi_1+\cdots+\xi_n=\sum_{j=1}^nE(B_j)\xi.
\]
Thus $E(A_n)=\sum_{j=1}^nE(B_j)$ and in order to prove property \ref{def:idemspecmeas}\eqref{it:ivsmiii}, it will suffice to show
that $E(A_n)$ converges to $E(A)$ in strong operator topology as $n\to\infty$.
Given $\xi\in\HEu$, we have $\xi=\xi_0+\xi_1$ for $\xi_0\in P(T,A^c)\HEu$ and $\xi_1\in P(T,A)\HEu$.
Then for all $n$, $E(A_n)\xi_0=E(A)\xi_0=0$.
Since $P(T,A_n)$ increases and converges in strong operator topology to $P(T,A)$ as $n\to\infty$,
the vector $\xi_1^{(n)}:=P(T,A_n)\xi_1$ converges to $P(T,A)\xi_1=\xi_1$ as $n\to\infty$.
Let $\eps>0$.
For all $n$ sufficiently large, we have
\[
\|\xi_1^{(n)}-\xi_1\| < \frac{\epsilon}{1+\sup_{B} \|E(B)\|}
\]
and for such $n$ we have
\begin{align*}
\|E(A_n)\xi-E(A)\xi\|&=\|E(A_n)\xi_1-\xi_1\| \\
&\le\|E(A_n)(\xi_1-\xi_1^{(n)})\|+\|E(A_n)\xi_1^{(n)}-\xi_1\| \\
&=\|E(A_n)(\xi_1-\xi_1^{(n)})\|+\|\xi_1^{(n)}-\xi_1\| \\
&\le(\|E(A_n)\|+1)\|\xi_1^{(n)}-\xi_1\|<\eps.
\end{align*}
This completes the proof of property \ref{def:idemspecmeas}\eqref{it:ivsmiii}.

We now prove~\eqref{it:TEB}.
Given $\xi\in\HEu$, we write $\xi=\xi_0+\xi_1$ where $\xi_0\in P(T,B^c)\HEu$ and $\xi_1\in P(T,B)\HEu$.
Since $E(B)\HEu = P(T,B)\HEu$ and $P(T,B^c)\HEu$ are invariant subspaces for $T$, we have
\[
E(B)T\xi=E(B)T(\xi_0+\xi_1)=E(B)T\xi_1=T\xi_1=TE(B)(\xi_0+\xi_1)=TE(B)\xi.
\]
This proves that $T$ and $E(B)$ commute.

The assertion~\eqref{it:EBBrown} follows immediately from $E(B)\HEu=P(T,B)\HEu$ and the property of Haagerup--Schultz projections. 
\end{proof}

\begin{thm}\label{thm:unzaspec}
Let $T\in \Mcal$. Then the following are equivalent:
\begin{enumerate}[(a)]
\item\label{it:unzaspec.a} $T$ has the UNZA property,
\item\label{it:unzaspec.b} there exist $S,Q\in\Mcal$ with $[S,Q]=0$, $S$ a scalar type operator and $Q$ s.o.t.-quasinilpotent, such that $T=S+Q$,
\item\label{it:unzaspec.c} there exist $A,N,Q' \in \Mcal$, with $[N,Q']=0$, $N$ normal, $Q'$ s.o.t.-quasi\-nil\-po\-tent, and $A$ invertible, such that $ATA^{-1}=N+Q'$.
\end{enumerate}
\end{thm}
\begin{proof}
\eqref{it:unzaspec.a}$\implies$\eqref{it:unzaspec.b}.
Assume $T$ has the UNZA property.
Using the spectral measure $E$ from Lemma~\ref{lem:unza-spc}, define
\[ 
S = \int_{\Cpx} \lambda E(d\lambda).
\]
This integral exists, and $S$ is a bounded operator, since, by construction, $E(B)=0$ for $B \subset \sigma(T)^c$.
Moreover, by definition, $S$ an operator of scalar type.

Let $Q = T - S$. Since $TE(B)=E(B)T$, it follows that $Q$ and $S$ commute. 
We claim that $Q$ is s.o.t.-quasinilpotent. 
Using Proposition~\ref{prop:spcproj} and Lemma~\ref{lem:unza-spc},
for every $B \in \Afr$ we have
\[
P(S,B)\HEu=E(B)\HEu=P(T,B)\HEu,
\]
so the Haagerup--Schultz projections of $S$ and $T$ agree.
Using the pushforward result from Proposition~\ref{prop:hscomm}, we have
\[ 
  P(Q,B) = P( (T,S):\{(\lambda_1,\lambda_2) : \lambda_1 - \lambda_2 \in B\} ) 
\] 
Since $P(T,\cdot)=P(S,\cdot)$, and $P( (T,S) : B_1 \times B_2 ) = P(T,B_1) \wedge P(T,B_2)$, it follows that 
$\mu_{(T,S)}$ is concentrated on the set $\{(z,z): z \in \Cpx \}$. 
Hence, if $0 \notin B$, then $P(Q,B)=0$.
Thus implies that $Q$ is s.o.t.-quasinilpotent.

\eqref{it:unzaspec.b}$\implies$\eqref{it:unzaspec.c}.
Assuming $T=S+Q$ with $S$ of scalar type and $Q$ s.o.t.-quasinilpotent and commuting with $S$,
let $A\in\Mcal$ be the invertible operator from Theorem~\ref{thm:norm-scl} so that $N:=ASA^{-1}$ is normal.
Let $Q'=AQA^{-1}$.
Since similarity doesn't change the Brown measure, we have that $Q'$ is s.o.t.-quasinilpotent.
Moreover, $N$ and $Q'$ commute.
We have $ATA^{-1}=N+Q'$, as required.

\eqref{it:unzaspec.c}$\implies$\eqref{it:unzaspec.a}. Assume $ATA^{-1}=N+Q'$ as described in~\eqref{it:unzaspec.c}.
Let $B \in \Afr$. Then, from Theorem \ref{thm:hsproj-sim} and Proposition~\ref{prop:hscomm}, we get
 \begin{equation}\label{eq:unza1}
  P(N,B) \HEu = P(ATA^{-1},B) \HEu = AP(T,B)\HEu 
 \end{equation}
If $T$ fails to have the UNZA property, there exist sets $B_n$ so that 
 \[ \alpha\left(P(T,B_n),P(T,B_n^c) \right) \to 0 \]
Then, there exist $v_n \in P(T,B_n)\HEu$, $w_n \in P(T,B_n^c)\HEu$ such that $ \left\langle v_n, w_n \right \rangle \to 1$, and $\norm{v_n}=\norm{w_n}=1$. 
Since $N$ is normal, its spectral subspaces are orthogonal. So, from~\eqref{eq:unza1}, we have
\begin{equation}
\norm{A(v_n-w_n)}^2 = \norm{Av_n}^2 + \norm{Aw_n}^2 \geq 2 \norm{A^{-1}}^{-2} > 0,
\end{equation}
which contradicts the fact that $\norm{v_n - w_n} \to 0$. 
\end{proof} 

\begin{cor}
Let $T \in \Mcal$.
Then $P(T,\cdot)$ defines a spectral measure if and only if $T=N+Q$ for some $N,Q\in\Mcal$, where $N$ is normal, $Q$ is s.o.t.-quasinilpotent, and $NQ=QN$.
\end{cor}
\begin{proof}
If $T=N+Q$ as described, then, from Proposition~\ref{prop:hscomm}, $P(T,\cdot)=P(N,\cdot)$ is a spectral measure.
On the other hand, if $P(T,\cdot)$ is a spectral measure, $P(T,B)P(T,B^c)=0$, and hence the corresponding subspaces are orthogonal, and so $T$ has the UNZA property. The construction
\eqref{it:unzaspec.a}$\implies$\eqref{it:unzaspec.b} in the proof of Theorem~\ref{thm:unzaspec} then yields $T=S+Q$ with $S$ actually normal.
\end{proof}

It is well known and is also easily seen from the above that spectral operators are decomposable.
With the help of Theorem \ref{thm:unzaspec}, we get the following equivalance:

\begin{cor}\label{cor:specdec}
Let $T \in \Mcal$. Then, $T$ is spectral if and only if $T$ is decomposable and satisfies the UNZA property.
\end{cor}
\begin{proof}
It is clear from definitions that spectrality implies decomposability and, from Theorem~\ref{thm:unzaspec}
and the characterization in Proposition~\ref{prop:spcdec}, that spectrality implies the UNZA property.

To prove the converse, suppose that $T$ is decomposable and has the UNZA property.
Let $E$ be the idempotent-valued spectral measure constructed in Lemma~\ref{lem:unza-spc}.
By decomposability and Proposition~\ref{prop:decom}, we see that for each $B \in \Afr$,
the spectrum of the restriction of $T$ to $E(B)$ is contained in the closure of $B$.
Thus, $T$ is spectral.
\end{proof}

\section{Some non-spectral but strongly decomposable operators}
\label{sec:circ}

The following simple example constructs an operator which is decomposable (even Borel and hence strongly decomposable), but not spectral.

\begin{example}
By a standard construction, we can realize the von Neumann algebra direct sum 
\[ 
  \Mcal = \bigoplus_{n=1}^\infty M_{2} (\Cpx)
\]
as a von Neumann subalgebra of the hyperfinite II$_1$ factor. 

Let 
\begin{equation}\label{eq:Toplus}
T  = \bigoplus_{n=1}^\infty \left(\begin{matrix} 0 & 1 \\ 0 & 1/n \end{matrix} \right) 
\end{equation}
We claim that $\sigma(T)=\{1,1/2,1/3...\} \cup \{0\}$.
This is easily seen, for given $\lambda \notin \{0\} \cup \{1/n\}_{n=1}^\infty$, we have
\[
\left( \begin{matrix} \lambda & -1 \\ 0 & \lambda - 1/n \end{matrix} \right)^{-1} = (\lambda^2 - \lambda/n)^{-1} \left( \begin{matrix} \lambda - 1/n & 1 \\ 0 & \lambda \end{matrix} \right) 
\]
and this is uniformly bounded in norm as $n\to\infty$.

It is well known that every operator (like $T$) with countable spectrum is decomposable.
In fact (see~\cite{DNZ18}) it is even Borel decomposable, which is a stronger condition than strong decomposability.

The vector $(1,0)^t$ is an eigenvector with eigenvalue $0$ for each matrix block in~\eqref{eq:Toplus}.
For each $n$, the vector $(1,1/n)^t$ is the other eigenvector of the $n$th matrix block in~\eqref{eq:Toplus}, with eigenvalue $1/n$.
But the angle between $(1,1/n)^t$ and $(1,0)^t$ goes to $0$ as $n\to\infty$.
This implies
\[
\alpha( P(T,\{0\}),P(T,\Cpx \setminus \{0\})) = 0.
\]
So the operator $T$ fails to have the non-zero angles property and, by Corollary~\ref{cor:specdec}, is not spectral.

This concludes the example.
\end{example}

\medskip

The rest of this section is devoted to showing that no circular free Poisson operator is spectral.  This includes the case of Voiculescu's circular operator.
   
\begin{thm}\label{thm:circnonspec}
Let $Z$ be a circular free Poisson operator. Then $Z$ fails to have the non-zero angle property and, thus, is not spectral.
\end{thm}
\begin{proof}
Let $Z$ be a circular free Poisson operator of parameter $c\ge1$.
Thus (see Section~\ref{subsec:cfP}), $Z$ is a $DT(\mu,1)$ operator, where $\mu$ is uniform measure on the annulus $A_{\sqrt{c-1},\sqrt{c}}$ and, moreover, $Z$ is R-diagonal.
Let $N\in\Nats$, $N > 2$.
By Theorem 4.12 of~\cite{DH04}, we may realize $Z$ in $M_N(\Mcal)\cong M_N(\Cpx)\otimes\Mcal$ (with respect to the tracial state $\frac1N\Tr_N\otimes\tau$) as an upper triangular
matrix
\[
Z=\begin{pmatrix}
a_1 & b_{12} & \cdots & \cdots & b_{1N} \\
0 & a_2 & b_{23} & & \vdots \\
\vdots & \ddots & \ddots & \ddots & \vdots\\
\vdots && 0 & a_{N-1} & b_{N-1,N} \\[1.5ex]
0& \cdots & \cdots & 0 & a_N
\end{pmatrix},
\]
With
\begin{itemize}
\item[$\bullet$] $(a_k)_{k=1}^N,\,(b_{ij})_{1\le i<j\le N}$ a $*$-free family in $(\Mcal,\tau)$
\item[$\bullet$] each $b_{ij}$ circular with $\tau(b_{ij}^*b_{ij})=\frac1N$,
\item[$\bullet$] each $a_j$ a $DT(\mu_j,\frac1{\sqrt{N}})$ operator for a Borel probability measure $\mu_j$ on $\Cpx$,
\item[$\bullet$] $\frac1N(\mu_1+\cdots+\mu_N)=\mu$
\end{itemize}
and we are free to choose the $\mu_j$ subject to these conditions.
In particular, we may choose a measurable partition of the annulus $A_{\sqrt{c-1},\sqrt c}$ into $N$ equally weighted, pairwise disjoint sets, $E_1,\ldots,E_N$
and let $\mu_j$ be the renormalized restriction of $\mu$ to $E_j$.
Let $\delta_1,\delta_2$ be such that 
\begin{equation}\label{eq:delts}
\frac1N<\delta_2<\delta_1<1-\frac1N.
\end{equation}
We may choose such a partition $E_1,\ldots,E_N$ so that
\begin{equation}\label{eq:E1E2}
E_1=A_{\sqrt{c-\delta_1-\frac1N},\sqrt{c-\delta_1}},\qquad
E_2=A_{\sqrt{c-\delta_2},\sqrt{c-\delta_2+\frac1N}}.
\end{equation}
Then $\sqrt Na_j$ is a $\DT(\mut_j,1)$-operator, where
$\mut_1$ and $\mut_2$ are uniform measures on
$A_{\sqrt{N(c-\delta_1)-1},\sqrt{N(c-\delta_1)}}$ and $A_{\sqrt{N(c-\delta_2)},\sqrt{N(c-\delta_2)+1}}$, respectively.
Namely, $\sqrt{N}a_1$ and  $\sqrt{N}a_2$ are circular free Poisson of parameters $N(c-\delta_1)$ and $N(c-\delta_2)+1$, respectively.
In partiular, $a_2$ is invertible and, by Proposition~\ref{prop:cfPnorms}, we have 
\begin{align}
\|a_1\|_2^2&=\frac1N(N(c-\delta_1))=c-\delta_1 \label{eq:a1nm} \\
\|a_2^{-1}\|_2^2&=\frac N{N(c-\delta_2)}=\frac1{c-\delta_2}. \label{eq:a2nm}
\end{align}

The upper left $2\times 2$ corner of $Z^n$ is equal to 
\[
\begin{pmatrix}
a_1^n & \sum_{k=0}^{n-1}a_1^kb_{12}a_2^{n-k-1} \\[1ex]
0 & a_2^n
\end{pmatrix}.
\]
We regard $M_N(\Mcal)$ as acting on the Hilbert space $\HEu:=L^2(\Mcal,\tau)^{\oplus N}$, whose elements are thought of as column vectors of length $N$.
For each $n\in\Nats$, let $\eta_n,\xi_n\in L^2(\Mcal,\tau)$ be the elements
\[
\eta_n=(a_2^{-n})\hat{\;},\qquad\xi_n=(b_{12}a_2^{-n-1})\hat{\;}
\]
and let $v_n=(\xi_n,\eta_n,0,0,\ldots,0)^t\in\HEu$.
Then
\[
Z^nv_n=\left({\textstyle\sum_{k=0}^n(a_1^kb_{12}a_2^{-k-1})\hat{\;}},\,\oneh,\,0,0,\ldots,0\right)^t.
\]
Now for all $k\ge0$, we have
\begin{align}
\|a_1^kb_{12}a_2^{-k-1}\|_2^2&=\tau((a_2^{-k-1})^*b_{12}^*(a_1^k)^*a_1^kb_{12}a_2^{-k-1})  \label{eq:abak} \\
&=\tau((a_1^k)^*a_1^k)\tau(b_{12}^*b_{12})\tau((a_2^{-k-1})^*a_2^{-k-1}) \notag \\
&=\frac{\|a_1^k\|_2^2\|a_2^{-k-1}\|_2^2}N
=\frac{\|a_1\|_2^{2k}\|a_2^{-1}\|_2^{2(k+1)}}N \notag \\
&=\frac{(c-\delta_1)^k}{N(c-\delta_2)^{k+1}}=\frac1{N(c-\delta_2)}\left(\frac{c-\delta_1}{c-\delta_2}\right)^k, \notag 
\end{align}
where the second equality is a result of $*$-freeness, the fourth equality is from Proposition~\ref{prop:Rdiag}, and the fifth is from~\eqref{eq:a1nm} and~\eqref{eq:a2nm}.
Thus, we may set
\[
\zeta=\sum_{k=0}^\infty(a_1^kb_{12}a_2^{-k-1})\hat{\;}\in L^2(\Mcal,\tau),
\]
where convergence is with respect to $\|\cdot\|_2$, and we have
\[
x:=\lim_{n\to\infty}Z^nv_n=(\zeta,\oneh,0,0,\ldots,0)^t\in\HEu.
\]
On the other hand, in a similar manner we have, in $\HEu$,
\begin{align*}
\|v_n\|^2&=\|\eta_n\|^2+\|\xi_n\|^2 \\
&=\tau((a_2^{-n})^*a_2^{-n})+\tau((a_2^{-n-1})^*b_{12}^*b_{12}a_2^{-n-1}) \\
&=\|a_2^{-n}\|_2^2+\tau(b_{12}^*b_{12})\tau((a_2^{-n-1})^*a_2^{-n-1}) \\
&=\|a_2^{-1}\|_2^{2n}+\frac1N\|a_2^{-1}\|_2^{2(n+1)} \\
&=\frac1{(c-\delta_2)^n}+\frac1{N(c-\delta_2)^{n+1}}
\end{align*}
and
\[
\lim_{n\to\infty}\|v_n\|^{1/n}=\frac1{\sqrt{c-\delta_2}}.
\]
By Proposition~\ref{prop:disks}, this implies that $x$ lies in the range of the 
Haagerup--Schultz projection $P(Z,A_{\sqrt{c-\delta_2},\infty})$.
However, using Proposition~5.3 of~\cite{DH04},
we have the following inclusion involving local spectral subspaces:
\[
L^2(\Mcal,\tau)^{\oplus 2}\oplus 0^{\oplus N-2}\subseteq\HEu_Z(E_1\cup E_2),
\]
where $E_1$ and $E_2$ are as in~\eqref{eq:E1E2}.
Thus, by Proposition~\ref{prop:HSdecomp}, $x$ belongs to the range of $P(Z,E_1\cup E_2)$.
Using the lattice properties of the Haagerup--Schultz projections (Theorem~\ref{thm:hsproj-lattice}),
we have that $x$ belongs to the range of
\[
P(Z,E_1\cup E_2)\wedge P(Z,A_{\sqrt{c-\delta_2},\infty})=P(Z,E_2).
\]
Again using Proposition~5.3 of~\cite{DH04}, we have
\[
L^2(\Mcal,\tau)\oplus 0^{\oplus N-1}\subseteq\HEu_Z(E_1).
\]
In particular, we have that
\[
y:=(\zeta,\underset{N-1\textrm{ times}}{\underbrace{0,\ldots,0}})^t
\]
lies in this range of $P(Z,E_1)$.
Let $\theta$ be the angle between the vectors $x$ and $y$.
Then
\[
\cos\theta=\frac{|\langle x,y\rangle|}{\|x\|\,\|y\|}=\frac{\|\zeta\|}{\sqrt{1+\|\zeta\|^2}}.
\]
Forcing $\theta$ to be arbitrarily close to zero is equivalent to forcing $\|\zeta\|$ to be arbitrarily large.
We compute
\begin{align*}
\|\zeta\|^2&=\sum_{k,\ell\ge0}\tau((a_2^{-k-1})^*b_{12}^*(a_1^k)^*a_1^\ell b_{12}a_2^{-\ell-1}) \\
&=\sum_{k,\ell\ge0}\tau((a_1^k)^*a_1^\ell)\tau(b_{12}^*b_{12})\tau((a_2^{-k-1})^*a_2^{-\ell-1}) \displaybreak[1] \\
&=\sum_{k=0}^\infty\tau((a_1^k)^*a_1^k)\tau(b_{12}^*b_{12})\tau((a_2^{-k-1})^*a_2^{-k-1}) \displaybreak[2] \\
&=\sum_{k=0}^\infty\frac{\|a_1^k\|_2^{2}\|a_2^{-k-1}\|_2^2}N
=\sum_{k=0}^\infty\frac{\|a_1\|_2^{2k}\|a_2^{-1}\|_2^{2(k+1)}}N \displaybreak[1] \\
&=\sum_{k=0}^\infty\frac{(c-\delta_1)^k}{N(c-\delta_2)^{k+1}}
=\frac1{N(c-\delta_2)}\sum_{k=0}^\infty\left(\frac{c-\delta_1}{c-\delta_2}\right)^k=\frac1{N(\delta_1-\delta_2)},
\end{align*}
where for the second equality we used $*$-freeness of $(a_1,b_{12},a_2)$, for the third equality we used R-diagonality of $a_1$, and the remaining part
of the computation follows as in~\eqref{eq:abak}.

To summarize, we have shown that given $\eps>0$ and $N\ge3$, by choosing $\delta_2$ and $\delta_1$ satisfying~\eqref{eq:delts} and so that $N(\delta_1-\delta_2)$
is sufficiently small, we ensure
\[
\alpha\left(P\left(Z,A_{\sqrt{c-\delta_1-\frac1N},\sqrt{c-\delta_1}}\right),P\left(Z,A_{\sqrt{c-\delta_2},\sqrt{c-\delta_2+\frac1N}}\right)\right)<\eps.
\]
This already implies that $Z$ fails to have the UNZA property.
It is now, however, an easy matter to show that $Z$ also fails to have the NZA property.
We recursively choose $N_1<N_2<\cdots$ and $\delta_2^{(k)},\,\delta_1^{(k)}$ satisfying
\begin{equation}\label{eq:deltk}
\frac1{N_k}<\delta_2^{(k)}<\delta_1^{(k)}<1-\frac1{N_k}
\end{equation}
and so that, letting
\begin{align*}
E_1^{(k)}&=A_{\sqrt{c-\delta_1^{(k)}-\frac1{N_k}},\sqrt{c-\delta_1^{(k)}}} \\
E_2^{(k)}&=A_{\sqrt{c-\delta_2^{(k)}},\sqrt{c-\delta_2^{(k)}+\frac1{N_k}}},
\end{align*}
the annuli
\begin{equation}\label{eq:Es}
E_1^{(1)},E_2^{(1)},\ldots,E_1^{(k)},E_2^{(k)},\ldots
\end{equation}
are pairwise disjoint, and
\begin{equation}\label{eq:Ndelt0}
\lim_{k\to\infty}N_k(\delta_1^{(k)}-\delta_2^{(k)})=0.
\end{equation}
This will ensure
\[
\lim_{k\to\infty}\alpha(P(Z,E_1^{(k)}),P(Z,E_2^{(k)}))=0
\]
so that, letting $F_j=\bigcup_{k=1}^\infty E_j^{(k)}$, we will have $F_1\cap F_2=\emptyset$ and
\[
\alpha(P(Z,F_1),P(Z,F_2))=0.
\]
This will imply that $Z$ fails the NZA property.

To see that the recursive choices of $N_k$, $\delta_1^{(k)}$ and $\delta_2^{(k)}$ may be made,
we start with $N_1=3$, $\delta_2^{(1)}=\frac49$ and $\delta_1^{(1)}=\frac59$.
Suppose $N_k$, $\delta_2^{(k)}$ and $\delta_1^{(k)}$ have been chosen.
We note that the stipulation~\eqref{eq:deltk} implies that the annulus $E_2^{(k)}$ lies outside of the annulus $E_1^{(k)}$.
We will choose
\begin{equation}\label{eq:k+1}
N_{k+1},\quad\delta_2^{(k+1)},\quad\delta_1^{(k+1)}
\end{equation}
so that $E_1^{(k+1)}$ lies outside of $E_2^{(k)}$, which will ensure pairwise disjointness of the annuli~\eqref{eq:Es}.
For this, we will need
\[
c-\delta_1^{(k+1)}-\frac1{N_{k+1}}>c-\delta_2^{(k)}+\frac1{N_k},
\]
This will hold if we choose the quantities~\eqref{eq:k+1} so that 
\[
\frac1{N_{k+1}}<\delta_2^{(k+1)}<\delta_1^{(k+1)}<\delta_2^{(k)}-\frac1{N_k}-\frac1{N_{k+1}}
\]
holds.
This is possible.
Indeed, we first choose $N_{k+1}$ so that
\[
\frac2{N_{k+1}}<\delta_2^{(k)}-\frac1{N_k},
\]
and then we choose $\delta_1^{(k+1)}$ satisfying
\[
\frac1{N_{k+1}}<\delta_1^{(k+1)}<\delta_2^{(k)}-\frac1{N_k}-\frac1{N_{k+1}}
\]
and, finally, we choose $\delta_2^{(k+1)}$ satisfying
\[
\frac1{N_{k+1}}<\delta_2^{(k+1)}<\delta_1^{(k+1)}
\]
and, furthemore, 
$N_{k+1}(\delta_1^{(k+1)}-\delta_2^{(k+1)})<\frac1k$,
in order to ensure that~\eqref{eq:Ndelt0} holds.
\end{proof}


\begin{thebibliography}{19}

\bib{A68}{article}{
   author={Apostol, Constantin},
   title={Spectral decompositions and functional calculus},
   journal={Rev. Roumaine Math. Pures Appl.},
   volume={13},
   date={1968},
   pages={1481--1528},
}

\bib{B83}{article}{
  author={Brown, Lawrence G.},
  title={Lidskii's theorem in the type II case},
  conference={
    title={Geometric methods in operator algebras},
    address={Kyoto},
    date={1983},
  },
  book={
    series={Pitman Res. Notes Math. Ser.}, 
    volume={123},
    publisher={Longman Sci. Tech.},
    address={Harlow},
    date={1986},
  },
  pages={1--35},
}

\bib{CDSZ17}{article}{
   author={Charlesworth, Ian},
   author={Dykema, Ken},
   author={Sukochev, Fedor},
   author={Zanin, Dmitriy},
   title={Simultaneous upper triangular forms for commuting operators in a finite von Neumann algebra},
   journal={Canad. J. Math.},
   volume={72},
   date={2020},
   number={5},
   pages={1188-1245},
}
\bib{D54}{article}{
   author={Dunford, Nelson},
   title={Spectral operators},
   journal={Pacific J. Math.},
   volume={4},
   date={1954},
   pages={321--354},
}

\bib{D58}{article}{
    author={Dunford, Nelson},
    title={A survey of the theory of spectral operators},
    journal={Bulletin of the AMS},
    volume={64.5},
    date={1958},
    pages={217-274},
}

\bib{DH01}{article}{
   author={Dykema, Ken},
   author={Haagerup, Uffe},
   title={Invariant subspaces of Voiculescu's circular operator},
   journal={Geom. Funct. Anal.},
   volume={11},
   date={2001},
   pages={693--741},
}

\bib{DH04}{article}{
   author={Dykema, Ken},
   author={Haagerup, Uffe},
   title={DT-operators and decomposability of Voiculescu's circular operator},
   journal={Amer. J. Math.},
   volume={126},
   date={2004},
   pages={121--189},
}
		
\bib{DNZ18}{article}{
   author={Dykema, Ken},
   author={Noles, Joseph},
   author={Zanin, Dmitriy},
   title={Decomposability and norm convergence properties in finite von Neumann algebras},
   journal={Integral Equations Operator Theory},
   volume={90},
   date={2018},
   pages={Art. 54, 32},
}

\bib{DSZ15}{article}{
   author={Dykema, Ken},
   author={Sukochev, Fedor},
   author={Zanin, Dmitriy},
   title={A decomposition theorem in II$_1$--factors},
   journal={J. reine angew. Math.},
   volume={708},
   date={2015},
   pages={97--114},
}

\bib{F63}{article}{
   author={Foia{\c{s}}, Ciprian},
   title={Spectral maximal spaces and decomposable operators in Banach space},
   journal={Arch. Math. (Basel)},
   volume={14},
   date={1963},
   pages={341--349},
}

\bib{F68}{article}{
   author={Foia{\c{s}}, Ciprian},
   title={Spectral capacities and decomposable operators},
   journal={Rev. Roumaine Math. Pures Appl.},
   volume={13},
   date={1968},
   pages={1539--1545},
}

\bib{HL01}{article}{
   author={Haagerup, Uffe},
   author={Larsen, Flemming},
   title={Brown's spectral distribution measure for $R$-diagonal elements in finite von Neumann algebras},
   journal={J. Funct. Anal.},
   volume={176},
   date={2000},
   pages={331--367},
}

\bib{HS09}{article}{
  author={Haagerup, Uffe},
  author={Schultz, Hanne},
  title={Invariant subspaces for operators in a general II$_1$--factor},
  journal={Publ. Math. Inst. Hautes \'Etudes Sci.},
  volume={109},
  year={2009},
  pages={19-111},
}

\bib{LN00}{book}{
   author={Laursen, Kjeld B.},
   author={Neumann, Michael M.},
   title={An introduction to local spectral theory},
   series={London Mathematical Society Monographs. New Series},
   volume={20},
   publisher={The Clarendon Press, Oxford University Press, New York},
   date={2000},
}

\bib{M59}{article}{
   author={Mackey, George W.},
   title={Commutative Banach algebras},
   journal={Notas Mat. No.},
   volume={17},
   date={1959},
   pages={210},
}

\bib{NS97}{article}{
  author={Nica, Alexandru},
  author={Speicher, Roland},
  title={$R$-diagonal pairs---a common approach to Haar unitaries and circular elements},
  conference={
    title={Free Probability Theory},
    address={Waterloo, Ontario},
    date={1995},
  },
  book={
    series={Fields Inst. Commun.}, 
    volume={12},
    publisher={Amer. Math. Soc.},
    date={1997},
  },
  pages={149--188},
}

\bib{S69}{article}{
  title={A local spectral theory for operators},
  author={Stampfli, JG},
  journal={Journal of Functional Analysis},
  volume={4},
  number={1},
  pages={1--10},
  year={1969},
  publisher={Elsevier}
}

\bib{S06}{article}{
   author={Schultz, Hanne},
   title={Brown measures of sets of commuting operators in a type $\rm II_1$ factor},
   journal={J. Funct. Anal.},
   volume={236},
   date={2006},
   pages={457--489},
}

\bib{W54}{article}{
   author={Wermer, John},
   title={Commuting spectral measures on Hilbert space},
   journal={Pacific J. Math.},
   volume={4},
   date={1954},
   pages={355--361},
}


\end{thebibliography}
\end{document}